\documentclass{article}

\usepackage{amsmath,amsfonts,amssymb,amsthm}

\newtheorem{theorem}{Theorem}[section]
\newtheorem{corollary}[theorem]{Corollary}
\newtheorem{conjecture}[theorem]{Conjecture}
\newtheorem{lemma}[theorem]{Lemma}
\newtheorem{proposition}[theorem]{Proposition}
\theoremstyle{definition}

\title{On the Structure of the $h$-Vector of a Paving Matroid}
\author{ Criel Merino\thanks{Instituto de Matem\'aticas, Universidad
      Nacional Aut\'onoma de M\'exico, Area de la Investigaci\'on
      Cient\'{\i}fica, Circuito Exterior, C.U. Coyoac\'an 04510,
      M\'exico,D.F. M\'exico. e-mail:merino@matem.unam.mx. Supported
      by Conacyt of M\'exico Proyect 83977}, Steven
      D. Noble\thanks{Department of Mathematical Sciences,
      Brunel University, Kingston Lane, Uxbridge UB8 3PH,
      U.K. e-mail:mastsdn@brunel.ac.uk},
      Marcelino Ram\'irez-Iba\~nez\thanks{Instituto de Matem\'aticas, Universidad
      Nacional Aut\'onoma de M\'exico, Area de la Investigaci\'on
      Cient\'{\i}fica, Circuito Exterior, C.U. Coyoac\'an 04510,
      M\'exico,D.F. M\'exico. e-mail: marchelino@gmail.com},
      \and Rafael Villarroel\thanks{Centro de Investigaci\'on en Matem\'aticas, Universidad Aut\'onoma del Estado de Hidalgo, Carr. Pachuca-Tulancingo km. 4.5, Pachuca 42184 Hgo, Mexico. email: rafael@uaeh.edu.mx}
     }
\date{\today}

\begin{document}
\maketitle
\thispagestyle{empty}

\begin{abstract}
 We give two proofs that the $h$-vector of any paving matroid is a
 pure O-sequence, thus answering in the affirmative a conjecture made
 by R. Stanley, for this particular class of matroids.
 We also investigate the problem of obtaining good lower bounds for the number of bases of a paving
 matroid given its rank and number of elements.
 \end{abstract}
%%%%%%%%%%%%%%%%%%%%%%%%%%%%%%%%%%%%%%%%%%%%%%%%%%%%%%%%%%%%%%%%%%%
%%%%%%%%%%%%%% INTRODUCTION  %%%%%%%%%%%%%%%%%%%%%%%%%%%%%%%%%%%%%%
%%%%%%%%%%%%%%%%%%%%%%%%%%%%%%%%%%%%%%%%%%%%%%%%%%%%%%%%%%%%%%%%%%%
\section{Introduction}
\label{sec:intro}

 Matroids are important structures in combinatorics, particularly in
 relation to combinatorial optimization and graph theory,
 see~\cite{lawler:matroids,oxley:book,white:matroids3}.
 With any matroid $M$ there is an associated simplicial complex $\Delta(M)$ given by
 the independent
 sets of $M$. Such simplicial complexes are called \emph{matroid complexes} and are
 known to be \emph{shellable}, that is, the maximal faces are equicardinal and
 can be arranged in a certain order that helps inductive proofs. (We give a full definition of shellability in the next section.)
 One key combinatorial invariant associated with a shellable
 complex is its $h$-vector which encodes information such as, for example,
 its face and Betti numbers. For these reasons
  shellable
 complexes have received much attention,
 see~\cite{bjorner:shellable-cohen,bjorner+wachs:shellable-non-pure-I,bjorner+wachs:shellable-non-pure-II,van-tuyl+villarreal,ziegler}. The concept of shellability
 is also important in theoretical
 computer science as
 the entries of the $h$-vector of a graphic matroid $M(G)$ are the
 coefficients of the $H$-form of the reliability polynomial of the underlying graph $G$,
 see~\cite{chari+colbourn}.

  A non-void set of monomials $\mathcal{M}$ is a \emph{multicomplex}
  if whenever $m\in \mathcal{M}$ and $m'|m$, then $m'\in
  \mathcal{M}$. A finite
  or infinite
  sequence $h$=($h_0$, $h_1$, $\ldots$, $h_d$) of
  integers is called an  \emph{$O$-sequence} if there exists a multicomplex
  containing exactly $h_i$ monomials of degree $i$.
  An $O$-sequence is \emph{pure} if there exists a multicomplex containing $h_i$ monomials of degree $i$ such that all the maximal elements in the multicomplex have the same degree. Properties of pure O-sequences
  are mentioned in Section~\ref{sec:prelim}.

 In 1977, Richard Stanley made the following conjecture linking $h$-vectors of matroid complexes and $O$-sequences~\cite{stanley:cohen-macaulay}, (see also~\cite{stanley:combinatorics-commutative-algebra}).
  \begin{conjecture}\label{conj:Stanley}
   The $h$-vector of a matroid complex is a pure O-sequence.
 \end{conjecture}

No progress was made on this conjecture for some considerable time.
  But in
 1997,
 work of Norman Biggs~\cite{biggs:chip-firing,biggs:tutte-growth} together with~\cite{merino:chip-firing-Tutte} implicitly proved
 Conjecture~\ref{conj:Stanley} for cographic matroids.
 For an explicit exposition see~\cite{merino:chip-firing-matroid-complexes}. More recently, the
 conjecture was proved for
 rank two matroids in~\cite{stokes}, for lattice-path matroids
 in~\cite{schweig}, for cotransversal matroids in~\cite{oh} and most recently for rank three matroids in~\cite{ha+stokes+zanello:O-sequences}.

 A \emph{paving matroid} is one in which all circuits have size at least
 $r(M)$. Interest in paving matroids goes back to 1976 when
 Dominic Welsh~\cite{welsh:matroid} asked if most matroids are paving. This question
 was motivated by numerical results obtained in~\cite{blackburn+crapo+higgs:catalogue}, where a catalogue of all matroids with up to eight elements was presented. The numerical
 data was updated in~\cite{mayhew+royle:9} to include matroids with nine elements, and the results made the problem even
 more intriguing.  More
 recently, the
 authors of~\cite{mayhew+newman+welsh+whittle} conjecture that asymptotically
 almost every matroid is paving, that is, the proportion of $n$-element matroids which are paving tends to one as $n$ tends to infinity.

 In this work we give a proof that coloopless paving matroids satisfy
 Conjecture~\ref{conj:Stanley}. Should paving matroids genuinely form a significant proportion of all matroids, then our
 result will be of a different kind from all the previous work on Conjecture~\ref{conj:Stanley}, as all previous work only considers classes of matroids whose size is insignificant compared with the total number of matroids.

 This article is organized as follows. In Section~\ref{sec:prelim} we
 give definitions and basic properties of matroids, $h$-vectors and $O$-sequences. In the next section we prove Stanley's conjecture for
 paving matroids.
 The direct approach to
 Stanley's conjecture is to attempt to get a good bounds on the number of bases of a paving matroid in terms of its number of elements and rank and on the minimum number of elements in a pure multicomplex of degree $r$ in $d$ indeterminates which contains every monomial of degree $r-1$. This was our original approach to the problem but we were unable to obtain good enough explicit bounds. However, there appear to be some intriguing open questions concerning these problems including potential links with various other well-studied combinatorial objects.
A subclass of paving matroids, namely sparse paving matroids, was introduced by Jerrum in~\cite{jerrum:balanced} and has recently received attention in~\cite{mayhew+royle:9}. 
 In Section~\ref{sec:bounds} we obtain a good lower bound for the number of bases of a sparse paving in terms of the rank $r$ and number $n$ of elements. We have examples showing that this bound is tight for infinitely many values of $r$ and $n$.
We then move on to consider bounds on the sizes of pure multicomplexes of degree $r$ in $d$ indeterminates which contains every monomial of degree $r-1$ and conjecture a link with the number of aperiodic binary necklaces.
The last section
 contains our conclusions.

%%%%%%%%%%%%%%%%%%%%%%%%%%%%%%%%%%%%%%%%%%%%%%%%%%%%%%%%%%%%%%%%%%%
%%%%%%%%%%%%%%  PRELIMINARY  %%%%%%%%%%%%%%%%%%%%%%%%%%%%%%%%%%%%%%
%%%%%%%%%%%%%%%%%%%%%%%%%%%%%%%%%%%%%%%%%%%%%%%%%%%%%%%%%%%%%%%%%%%

\section{Preliminaries}
\label{sec:prelim}
In this section we introduce some definitions and key properties of shellable complexes
 and matroids. We assume some  familiarity with matroid theory.
 For an excellent exposition of
 shellability of matroid complexes see~\cite{bjorner:shellability} and for matroids
 see~\cite{oxley:book}.

 \subsection{$h$-vectors}

 Let $\Delta$ be a  \emph{simplicial complex} on the vertex set $V$ = $\{
 x_1,\ldots, x_n\}$. Thus,  $\Delta$ is a collection of subsets of $V$
 such that for all $i$, $\{x_i\}\in\Delta$, and if $F\in \Delta$
 and $F'\subseteq F$, then $F'\in \Delta$. The subsets in $\Delta$ are
 called \emph{faces} and the \emph{dimension} of a face with $i+1$ elements is
 $i$. The \emph{dimension} of $\Delta$ is the maximum dimension of a
 face in $\Delta$.  Associated with $\Delta$  we have
 its \emph{face vector}  or \emph{$f$-vector} $(f_0,f_1, \ldots,
 f_d)$,  where
 $f_i$ is the number of faces of size $i$ (or dimension $i-1$) in  $\Delta$.
 The \emph{face enumerator} is the generating function of the entries of the $f$-vector, defined by
 \[
     f_{\Delta}(x) = \sum_{i=0}^{d}{f_i x^{d-i}}.
 \]
 The maximal faces of $\Delta$ are called \emph{facets}.
When all these have the
 same cardinality $\Delta$ is said to be \emph{pure}.
 From now on we will only consider pure $d-1$-dimensional simplicial complexes.

Given a linear ordering $F_1, F_2, \ldots, F_t$ of the facets of a simplicial complex $\Delta$,
let $\Delta_i$ denote the subcomplex generated by the facets $F_1, F_2, \ldots, F_i$, that is, $F\in \Delta_i$ if and only if $F\in \Delta$ and $F\subseteq F_j$ for some $j$ with $1\leq j\leq i$.

For a pure simplicial complex $\Delta$ a \emph{shelling} is a linear order
of the facets $F_1, F_2, \ldots, F_t$ such that, for $2 \leq l
\leq t$,
\[ \{ F: F\subseteq F_l \text{ and } F\in \Delta_{l-1}\}\]
forms a pure $(\dim(\Delta)-1)$-dimensional simplicial complex. A complex is said to be \emph{shellable} if it is pure and admits a shelling.

Define, for $1 \leq l \leq t$, $R(F_l)= \{ x \in F_l\, |\,\ F_l\setminus
x \in \Delta_{l-1} \}$, where here $\Delta_0=\emptyset$. The number
of facets such that $|R(F_l)|=i$ is denoted by
$h_i$ and importantly does not depend on the particular shelling, see~\cite{bjorner:shellability}. The vector $(h_0, h_1, \ldots, h_d)$  is called
the $h$-{\it vector} of $\Delta$.
The \emph{shelling
polynomial} is the generating function of the entries of the $h$-vector, given by
 \[
     h_{\Delta}(x) = \sum_{i=0}^{d}{h_i x^{d-i}}.
 \]
  It is well known, see for example~\cite{bjorner:shellability}, that the
  face enumerator and the shelling polynomial satisfy the relation
  \[
          h_{\Delta}(x+1) =  f_{\Delta}(x)
  \]
  and so the coefficients satisfy
  \begin{align}\nonumber
     f_k &= \sum_{i=0}^{k}{h_i \binom{d-i}{k-i}}\\
\intertext{and}
  \label{eq:hk=fk}
     h_k &= \sum_{i=0}^{k}{(-1)^{i+k}f_i \binom{d-i}{k-i}},
  \end{align}
  for $0\leq k \leq d$.
%%
%% MATROID COMPLEX
%%
\subsection{Matroids and their complexes}

 A \emph{matroid}  is an ordered pair $M$=($E$, $\mathcal{I}$) such that $E$
 is a finite set and $\mathcal{I}$ is a collection of subsets of $E$
 satisfying the following three conditions:
 \begin{enumerate}
 \item $\emptyset \in \mathcal{I}$;
 \item if $I\in \mathcal{I}$ and $I'\subseteq I$, then $I'\in
 \mathcal{I}$;
 \item if $I_1$ and $I_2$ are in $\mathcal{I}$ and
 $|I_1|<|I_2|$, then there is an element $e \in I_2\setminus I_1$
 such that $I_1\cup \{e\}\in \mathcal{I}$.
 \end{enumerate}

 Maximal independent sets
 are called \emph{bases} and it follows easily from the conditions above that all bases have the same
 cardinality. This common cardinality is called the \emph{rank} of the matroid and is usually
 denoted by $r(M)$ or just $r$.

 One fundamental example is the class of \emph{uniform matroids}. The
 uniform matroid with rank $r$ and $n$ elements is denoted by $U_{r,n}$. A set of its elements is independent if and only if it has size at most $r$.

    We recall some basic definitions of matroid theory. A
minimal subset $C$ of $E$ that is not independent is called a \emph{circuit}. The
\emph{closure} $\overline{A}$ of a subset $A$ of $E$ is defined by
\[
    \overline{A}=A\cup \{a\,|\, M\text{ has a circuit } C \text{ such that }
                         a\in C\subseteq A\cup\{a\}\}.
\]
  A subset $S$ is \emph{spanning} if $\overline{S}=E$. A subset $H$
  is a \emph{hyperplane} if is a maximal non-spanning set. For all the
  other concepts of matroid theory we refer the reader to Oxley's
  book~\cite{oxley:book}.

 If $M$=($E$, $\mathcal{I}$) is a matroid, the family of all
 independent sets forms a simplicial complex of dimension
 $r(M)-1$, which we denote by $\Delta(M)$. The facets of $\Delta(M)$
 are the  bases of $M$  and therefore $\Delta(M)$ is
 pure. Complexes of this kind are called  {\it matroid complexes}.
 Matroid complexes are known to be shellable, see
 \cite{bjorner:shellability}.
%%  Also, the shelling polynomial of
%%  $\Delta(M)$ is an evaluation of the Tutte polynomial, see
%%  \cite{Bjorner}, that is,
%% \[
%%    T(M; x,1)= h_{\Delta(M)}(x)
%% \]
%%  and, by duality, we also have

%% \[
%%    T(M; 1,y)= h_{\Delta(M^*)}(y).
%% \]
%%  Also note that if $\chi(\Delta)$ is the Euler characteristic of $\Delta$,

%% \[
%%      (-1)^{r-1}\chi(\Delta(M)) = f_{\Delta(M)}(-1)=
%%      h_{\Delta(M)}(0) =h_r= T(M; 0,1),
%% \]
\emph{Loops} of a matroid are circuits of size one and consequently do not belong to any independent set.
Consequently they do not play any role in $\Delta(M)$ and so to investigate Conjecture~\ref{conj:Stanley}, we can
 safely just consider loopless matroids.

 Furthermore, \emph{coloops} of a matroid are elements
 contained in every basis. Equivalently, they belong to every facet of $\Delta(M)$.
 Suppose $M$ is formed from $M'$ by deleting a coloop. Then $r(M)=r(M')-1$ but more pertinently
  if the  $h$-vector of $M$ is $(h_0,
 h_1, \ldots, h_r)$, then the $h$-vector of $M'$ is $(h_0,
 h_1, \ldots, h_r, 0)$. Thus, all the relevant information concerning the $h$-vector of a matroid
 can still be obtained after
 deleting all its coloops. Consequently, for our purposes we only need to consider coloopless matroids.

\subsection{Pure O-sequences}
\label{multi}

%%  An {\it order ideal} (or {\it down-set}) of a poset $P$ is a subset
%%  $I$ of $P$
%%  such that if $x\in I$ and $y\leq x$, then $y\in I$. If we take as a
%%  poset $P$ the set of all monomials over indeterminates $z_1, \ldots,
%%  z_n$ and the order given by divisibility, then an order ideal of $P$
%%  is what we called early an  order ideal of monomials, or  {\it
%%  multicomplex},  over $z_1, \ldots, z_n$.

%%  If we form the poset $(\mathbb{N}^n, \leq)$, where $a \leq b$ if
%%  $a(i)\leq b(i)$ for $1\leq i \leq n$, then a multicomplex
%%  $\mathcal{M}$ can also be seen as an order ideal of  $(\mathbb{N}^n,
%%  \leq)$. More explicitly, if $\mathcal{M}$ is a multicomplex over
%%  $z_1, \ldots, z_n$, then the image of the function
%%  $\mu:\mathcal{M}\rightarrow \mathbb{N}^n$ defined by
%% \[
%%    \mu(z_1^{i_1}, \ldots, z_n^{i_n})= (i_1, \ldots,i_n)
%% \]
%%  is an order ideal of $(\mathbb{N}^n, \leq)$. So, we can use
%%  interchangeably both definitions of multicomplex.

  An explicit characterization of O-sequences can be
  found in \cite{stanley:cohen-macaulay}.
  However, a complete characterization is not known for pure
  O-sequences but Hibi~\cite{hibi} has shown that a pure
  O-sequence $(h_0, h_1, \ldots, h_d)$ must satisfy the following
  conditions:
  \begin{align}
           h_0&\leq h_1 \leq \cdots \leq h_{[d/2]}\label{eq:oseq1}\\
\intertext{and}
\label{eq:oseq2}
      h_i&\leq h_{d-i},  \text{ whenever } 0 \leq i \leq [d/2].
   \end{align}
   Hibi also conjectured that the $h$-vector of a matroid complex must
  satisfy inequalities~\eqref{eq:oseq1} and~\eqref{eq:oseq2}.

The following result concerning the $h$-vector of a matroid complex
  is due to Brown and Colbourn~\cite{brown+colbourn}.
    \begin{theorem}\label{th:Brown-Colbourn}
      The $h$-vector of a connected rank-d matroid satisfies the
      following inequalities:
      \begin{equation}
        \label{eq:bc}
          (-1)^j\sum_{i=0}^{j}{ (-b)^{i}h_i} \geq 0,\qquad 0\leq j \leq d,
       \end{equation}
      for any real number $b \geq 1$ with equality possible
      only if b=1.
    \end{theorem}
   This theorem shows that the converse of Stanley's conjecture is not true because the sequence $(1,4,2)$ is a pure $O$-sequence but does not satisfy the conditions of the theorem.

 Later, Chari~\cite{chari:decompositions} gave a stronger result that generalizes Theorem~\ref{th:Brown-Colbourn}
 and solves Hibi's conjecture.
  The fact that the $h$-vector of a coloop free matroid satisfies
  inequalities~\eqref{eq:oseq1}--\eqref{eq:bc} can also be proved~\cite{chari:inequalities}
  using the Tutte polynomial.

\section{Stanley's conjecture for paving matroids}
\label{sec:Stanleyconj}
A paving matroid $M=(E,\mathcal{I})$ is a matroid whose circuits all have size
at least $r(M)$. If $M$ is a rank-$r$ paving matroid, the face vector of $\Delta(M)$
 is easy to compute. Every subset of size $i<r$ is a face of
 $\Delta(M)$ and the  facets are the bases of $M$. Then, we get the
 following result which is implicit in~\cite{bjorner:shellability}.

 \begin{proposition}
  The $h$-vector of a rank-$r$ paving matroid with $n$ elements and
  $b(M)$ bases is $(h_0,\ldots,h_r)$ where $h_k=\binom{n-r+k-1}k$
  for $0\leq k\leq  r-1$ and $h_r=b(M)-\binom{n-1}{r-1}$.
 \end{proposition}
 \begin{proof}
   Using~\eqref{eq:hk=fk} and $f_i=\binom{n}{i}$ for $0\leq i\leq
   r-1$ we see that
  \[
    h_k=\sum_{i=0}^{k}{(-1)^{i+k} \binom{r-i}{k-i}}\binom{n}{i}.
  \]
  for $0\leq k\leq r-1$. Using the identity
  $(-1)^{a}\binom{b}{a}=\binom{a-b-1}{a}$ we  get
  \[
    h_k=\sum_{i=0}^{k}{\binom{k-r-1}{k-i} \binom{n}{i}}.
  \]
 Now using the Vandermonde convolution formula
 $\binom{a+b}{k}=\sum_{i=0}^{k}{\binom{a}{i} \binom{b}{k-i}}$ we get
 \[
   h_k=\binom{n-r+k-1}{k}.
 \]
 Because $\sum_{i=0}^{r}{h_i}=b(M)$, we get
 \[h_r=b(M)-\sum_{i=0}^{r-1}{h_i} = b(M) - \sum_{i=0}^{r-1}{\binom{n-r+i-1}{i}} = b(M) - \binom{n-1}{r-1}.\]
 \end{proof}

 The idea for proving that the $h$-vector of a coloopless
 paving matroid is the O-sequence of a pure multicomplex is simple. We
 define the multicomplex 
 $\mathcal{M}_{r,d}$ to be the pure multicomplex in which the maximal
 elements are all monomials of 
 degree
 $r$ in $d$ indeterminates $z_1,\ldots, z_{d}$. This multicomplex
 has $O$-sequence $(h_0,\ldots,h_{r})$, where $h_k=\binom{d+k-1}k$.

 Now, define the function
\[
      f(r,d)=\min\{ h_r\,|\, (h_0,\ldots h_r)\ \text{is the pure O-sequence
      of} \ \mathcal{M}\supset \mathcal{M}_{r-1,d}\}.
\]
This means that $f(r,d)$ is the minimum number of monomials of degree
$r$ in a pure multicomplex of degree $r$ which contains every monomial
of degree $r-1$ in the $d$ indeterminates $z_1,\ldots,z_d$. 
So for any positive integers $d$ and $r$, if $h_k = \binom{d+k-1}k$
for $0\leq k \leq r-1$ and $f(r,d) \leq h_r \leq \binom{d+r-1}{r}$,
the sequence $(h_0,h_1,\ldots,h_r)$ is a pure $O$-sequence. 

If $M$ is a paving matroid with $n$ elements and rank $r$, then by
taking $d=n-r$, we see that the $h$-vector of $M$ satisfies $h_k =
\binom{d+k-1}k$ for $0\leq k \leq r-1$. To prove Stanley's conjecture
for paving matroids, it will be sufficient to prove that $f(r,d) \leq
h_r \leq \binom{d+r-1}{r}$ or equivalently 
\[f(r,n-r) \leq b(M) - \binom{n-1}{r-1} \leq \binom{n-1}{r}.\]
The second inequality is trivial since $b(M)\leq \binom nr$, so we
focus on the first inequality.  

Some initial
values of $f$ are easy to get.
\begin{lemma}\label{lem:initial_values_f}
  For $r\geq 1$ and $d\geq 1$ we have that $f(1,d)=1$, $f(2,d)=\lceil
  d/2 \rceil$, $f(r,1)=1$ and $f(r,2)=\lceil r/2 \rceil$.
 \end{lemma}

 \begin{lemma}\label{lem:recursion_f(p,d)}
   $f(r,d)\leq f(r,d-1)+f(r-1,d)$.
 \end{lemma}
 \begin{proof}
Let $\mathcal{M}'$ be a multicomplex in indeterminates
$z_1,\ldots,z_{d-1}$ containing $\mathcal{M}_{r-1,d-1}$, having
$h$-vector $(h_0',\ldots,h_r')$ satisfying $h_r'=f(r,d-1)$. 
Let $\mathcal{M}''$ be a multicomplex in indeterminates
$z_1,\ldots,z_d$ containing $\mathcal{M}_{r-2,d}$, having $h$-vector
$(h_0'',\ldots,h_{r-1}'')$ satisfying $h_r''=f(r-1,d)$.

Consider the multicomplex $\mathcal{M}$ that is the
union of $\mathcal{M}'$ and \[z_d\mathcal{M}''=\{z_d\,m|\,m\in
    \mathcal{M}''\}.\] Then, $\mathcal{M}$ contains
 all
    the monomials over $z_1,\ldots, z_{d-1}$ of degree at most $r-1$ and
    all the monomials over $z_1,\ldots, z_{d}$ of degree at most $r-1$
    where $z_d$ has degree at least 1. These are precisely all the
    monomials over  $z_1,\ldots, z_{d}$ of degree at most
    $r-1$. Therefore $\mathcal{M}$ contains $\mathcal{M}_{r-1,d}$.

    It remains to prove that $\mathcal{M}$ is a multicomplex. Let
    $m\in \mathcal{M}$ and $m'|m$. Then $m'$ is a monomial in
    indeterminates $z_1,\ldots,z_d$ and either $m'=m$ or $m'$ has
    degree at most $r-1$. By using the previous part of the proof, in
    either case we obtain that $m \in \mathcal {M}$. 

%    If $m\in \mathcal{M}'$, then
 %   $z_d\not|m'$ and because $\mathcal{M}'$ is a multicomplex,
 %   $m'\in\mathcal{M}' \subset  \mathcal{M}$. If $m\in
 %   z_d\mathcal{M}''$ and $z_d|m'$, then $m'/z_d\in  \mathcal{M}''$
 %   and $m'\in  \mathcal{M}$. If $m\in z_d\mathcal{M}''$ but
 %   $z_d\not|m'$, then $m'$ is a monomial over indeterminates
 %   $z_1,\ldots, z_{d-1}$ of rank at most $p-1$, that is an element in
 %   $\mathcal{M}_{p-1,d-1}$ that is contained in  $\mathcal{M}'$.

    Finally, the O-sequence of  $\mathcal{M}$  is $(h'_0,
    h'_1+h''_0,\ldots, h'_r+h''_{r-1})$.

  \end{proof}

 Let $\mathcal{P}_{r,n}$ be the class of coloopless, loopless rank-$r$ paving
 matroids on $n$ elements. We define
 \[
      g(r,n)=\min\Big\{ b(M)-\binom{n-1}{r-1}\,|\, M\in \mathcal{P}_{r,n}\Big\}.
\]

 Observe that $g(r,n)$ equals the minimum value of $h_r$ among all $h$-vectors
 of matroids in $\mathcal{P}_{r,n}$. Thus, to prove
 Stanley's conjecture for paving matroids is enough to show that
 $g(r,n)\geq f(r,n-r)$.

  \begin{lemma}\label{lem:rank_1}
   For all $n\geq 1$, $g(1,n)\geq f(1,n-1)$.
 \end{lemma}
 \begin{proof}
    Up to isomorphism, the only matroid, in $\mathcal{P}_{1,n}$ is
   $U_{1,n}$, thus $g(1,n)=n-1$ and $f(1,n-1)=1$.
  \end{proof}

  \begin{lemma}\label{lem:rank_2}
    For all $n\geq 2$, $g(2,n)\geq f(2,n-2)$.
 \end{lemma}
 \begin{proof}
   It is enough to prove that for any loopless and coloopless rank-2
   paving matroid $M$ with $n$ elements, $b(M)\geq (n-1)+ \lceil (n-2)/2
   \rceil $. If every  element is in at least 3 bases, then $b(M)\geq
   3n/2> 3(n-1)/2\geq (n-1)+ \lceil (n-2)/2 \rceil $.

  Suppose that $M$ has an element $e$ in at most 2 bases.
  If the number of parallel classes in $M$ is at least 4, then every
  element is in at least 3 bases. Thus, $M$ has at most 3 parallel
  classes. If $M$ has two parallel classes, it is not possible that
  both have size at
  least 3, thus, there is a parallel classes of size 2. Therefore, $M$ is
  $U_{1,n-2}\oplus U_{1,2}$, where $n\geq 4$ as $M$ is coloopless, and $b(M)=2(n-2)\geq (n-1)+ \lceil (n-2)/2
  \rceil$, with equality when $n=4$.

  If there are three parallel classes, it is not possible that two
  have size at least 2, or else every element is in at least three bases. Thus there are two parallel classes of size one
  and one of size $n-2$. Therefore, $M$ is isomorphic to
  $U_{1,n-1}\oplus_2 U_{2,3}$, that is, the graphic matroid of a triangle with one edge replaced by $n-2$ parallel edges, and $b(M)=2(n-2)+1\geq 3(n-1)/2\geq
  (n-1)+ \lceil (n-2)/2 \rceil$ with equality when $n=3$ and $M$ is
  $U_{2,3}$.
  \end{proof}

 We use the following result from~\cite{chavez+merino+noble+ibanez}.
   \begin{lemma}\label{lem:2-stretching}
    Let $M$ be a rank-r coloopless paving matroid. If for every element $e$
    of $M$, $M\setminus e$ has a coloop, then one of the following
    three cases happens.
    \begin{enumerate}
    \item $M$ is isomorphic to $U_{r,r+1}$, $r\geq 1$.
    \item $M$ is the 2-stretching of a uniform matroid
       $U_{s,s+2}$, for some $s\geq 1$.
    \item $M$ is isomorphic to $U_{1,2}\oplus
    U_{1,2}$.
    \end{enumerate}
   \end{lemma}

  \begin{lemma}
    Let $M$ be a rank-r coloopless paving matroid with $n$
    elements. If for every element $e$ of $M$, $M\setminus e$ has a
    coloop, then $b(M)- \binom{n-1}{r-1}\geq f(r,n-r)$
 \end{lemma}
 \begin{proof}
    It follows from the previous lemma that we just have to check
    three cases. If $M\cong U_{r,r+1}$, $b(M)=r+1=\binom r{r-1}+ f(r,1)$. In this case we have equality.

    If the matroid $M$  is the
   2-stretching of $U_{s,s+2}$, it has rank $2s+2$, $2s+4$ elements and
   $2(s+2)(s+1)$  bases, thus $b(M)-\binom{2s+3}{2s+1}=s+1$. On the
   other hand, $f(2s+2,2)=s+1$, and we have equality.

Finally, if
    $M\cong U_{1,2}\oplus U_{1,2}$, then $M$ has rank 2 and has
    already been considered in Lemma~\ref{lem:rank_2}.
  \end{proof}

 \begin{theorem}\label{thm:main}
   For $r\leq n$ we have $g(r,n)\geq f(r,n-r)$.
 \end{theorem}
  \begin{proof}
    We prove the statement by induction on $r+n$. If
    $r\leq 2$ and $n\geq r$, the result follows by Lemmas~\ref{lem:rank_1}
    and~\ref{lem:rank_2}. Suppose that the statement is true for all $r'$ and $n'$ with $r'+n'<r+n$.

    Let $n\geq r>2$ and $M$ be a matroid in $\mathcal{P}_{r,n}$ such that
    $b(M)=\binom{n-1}{r-1}+g(r,n)$. Suppose that $M\setminus e$ has
    no coloops for some $e\in E(M)$. Then
    \begin{align*}
       g(r,n)&=b(M)- \binom{n-1}{r-1}
             =b(M\setminus e)-\binom{n-2}{r-1}+b(M/e)-\binom{n-2}{r-2}\\
             &\geq g(r,n-1)+g(r-1,n-1)\geq f(r,n-r-1)+f(r-1, n-r)\\&\geq f(r,n-r).
    \end{align*}
     If $M$ has no such element $e$ then the result follows by
     Lemma~\ref{lem:2-stretching}.
  \end{proof}

 \begin{corollary}\label{Stanley_Paving}
     The $h$-vector of the matroid complex of a paving matroid is a
     pure O-sequence.
 \end{corollary}

%%%%%%%%%%%%%%%%%%%%%%%%%%%%%%%%%%%%%%%%%%%%%%%%%%%%%%%%%%%%
%%%%%%%%%%%%%%%%% Second proof %%%%%%%%%%%%%%%%%%%%%%
%%%%%%%%%%%%%%%%%%%%%%%%%%%%%%%%%%%%%%%%%%%%%%%%%%%%%%%%%%%%

\section{Bounds on the number of bases of a sparse paving matroid}
One intriguing problem is to determine more about the functions $f$ and $g$ from the previous section. This appears to be a rather hard problem, in particular we have not been able to find tight bounds on the number of bases of a paving matroid in terms of its rank and number of elements. In this section we find a tight bound for the number of bases for a subclass of paving matroids, namely the sparse paving matroids. 

We will require the following result on paving matroids which is an Exercise in~\cite{oxley:book} (Page~132, Exercise~8).
 \begin{proposition}\label{paving_closed_minor}
   Paving matroids are closed under minors. Moreover a matroid $M$ is paving if and only if it does not contain the matroid
$U_{2,2}\oplus U_{0,1}$ as a minor.
 \end{proposition}

Sparse paving matroids were introduced by Jerrum in~\cite{jerrum:balanced,mayhew+royle:9}. 
A rank-$r$ matroid $M$ is {\em sparse paving}  if $M$ is paving and
  for every  pair of cycles $C_1$ and $C_2$ of size $r$ we have
  $|C_1\bigtriangleup C_2|>2$. For example, all uniform matroids are
  sparse paving matroids. The following lemma is straightforward.

 \begin{lemma}\label{sparse_rank_1}
  If $M$ is a sparse paving matroid with rank at most 1 and $n$
  elements, then $M$ is isomorphic to $U_{1,n}$, $U_{1,n-1}\oplus
  U_{0,1}$ or $U_{0,n}$.
 \end{lemma}

 \begin{theorem}\label{sparse_circuit_hyperplane}
   Let $M$ be a paving matroid of rank $r\geq 2$. Then $M$ is
   sparse paving if and only if all the hyperplanes of $M$ have size $r$
   or $r-1$ and the hyperplanes of size
   $r$ are  precisely its circuits of size $r$.
 \end{theorem}
 \begin{proof}
    For the forward implication let $H$ be a hyperplane of $M$ and $I$
    be a maximal independent set
    contained in $H$. If there are two elements $e\neq f$ in
    $H\setminus I$, then
    $C_1=I\cup\{e\}$ and $C_2=I\cup\{f\}$ are circuits of size $r$
    but $|C_1\bigtriangleup C_2|= 2$, contrary to the assumption that $M$
    is sparse paving. Thus, any hyperplane has size either $r-1$ or
    $r$ and in the case that $H$ has size $r$ it will also
    be a circuit. Let $C$ be a circuit of $M$ of size $r$. Then
    since $C$ has rank $r-1$, $\overline{C}$, the closure of $C$, is
    a hyperplane. By the previous argument this hyperplane has size
    $r$ and as $C\subseteq
    \overline{C}$, we conclude that $C$ is a hyperplane.

   To prove the converse, we take two circuits $C_1$ and $C_2$ of size
   $r$ in $M$. Then $I=C_1\cap C_2$ is an independent set and because
   $I$ is the intersection of two hyperplanes, its rank is at most
   $r-2$. So, $|C_1\cap C_2| \leq r-2$ and  $|C_1\bigtriangleup C_2|>
   2$.
 \end{proof}
The following result appears to be (recent) folklore but we are unable to find a reference.
 \begin{theorem}\label{sparse_closed_duality}
   If $M$ is an $n$-element sparse paving matroid, then $M^{*}$ is
   also sparse paving.
 \end{theorem}
 \begin{proof}
   If $M$ has rank at most 1, it follows by Lemma~\ref{sparse_rank_1}
   that $M^{*}$ is isomorphic to $U_{n-1,n}$, or $U_{n-2,n-1}\oplus
   U_{1,1}$ or $U_{n,n}$. In each case $M^{*}$ is sparse paving.

   Let us suppose that $M$ has rank $2\leq r \leq n-2$. By duality,
   $C$ is a circuit of a matroid $N$ over $E$ if and only if
   $E\setminus C$ is a hyperplane of $N^{*}$. 
   From
   Theorem~\ref{sparse_circuit_hyperplane} it follows that all the
   hyperplanes of $M$ have size $r$ or $r+1$. Consequently all the circuits of $M^{*}$ have size $n-r$ or $n-r-1$ and so $M^{*}$ is paving. 
   
   Furthermore all the hyperplanes of $M^{*}$ have size $n-r$ or $n-r-1$ and since hyperplanes and circuits of $M$ of size $r$ coincide, hyperplanes and circuits of $M^{*}$ also coincide and so by Theorem~\ref{sparse_circuit_hyperplane}, $M^{*}$ is sparse paving.
   
   Finally, if the rank of $M$ is $n-1$ or $n$, then $M$ and $M^{*}$ are
   uniform matroids and the result follows.
 \end{proof}

The next result was first proved by Jerrum~\cite{jerrum:balanced}. It follows immediately from Theorem~\ref{sparse_closed_duality} and the fact that the collection of circuits of
   $M\setminus e$ is the collection of circuits of $M$ that do not contain $e$.
 \begin{theorem}
   Sparse paving matroids are closed under minors
 \end{theorem}

  \begin{theorem}
      A matroid $M$ is sparse paving if and only if it does not have
   $U_{2,2}\oplus U_{0,1}$ nor $U_{0,2}\oplus U_{1,1}$ as  minors.
  \end{theorem}
  \begin{proof}
   By Theorem~\ref{paving_closed_minor}, it is enough to prove that a
   paving matroid $M$ is sparse if and only if $M$ does not contain
   $U_{0,2}\oplus U_{1,1}$ as a minor.

   If a rank-$r$ paving matroid $M$ contains $U_{0,2}\oplus U_{1,1}$ as a minor,
   then $M^{*}$ contains $U_{2,2}\oplus U_{0,1}$ as a minor and by Proposition~\ref{paving_closed_minor} it is not
   paving. Thus $M$ cannot be sparse paving by
   Theorem~\ref{sparse_closed_duality}.

   Suppose $M$ is a rank-$r$ paving matroid with $n$ elements that is not
   sparse. All hyperplanes of $M$ of size $r$ must be circuits because every set of size $r-1$ is independent. 
   Consequently there must exist
   a hyperplane $H$ of size at least $r+1$. 
   Let $I$ be a maximal indpendent set in $H$ and let $\{e,f\}\in H\setminus I$.
   Now let $g\not\in
   H$. If we delete the elements in $E\setminus (H\cup\{g\})$ and
   contract the elements in $H\setminus \{e,f\}$ we get a
   $U_{0,2}\oplus U_{1,1}$ minor.
  \end{proof}

 In order to get more properties of sparse paving matroids, we need
 the following definition from~\cite{oxley:book}. Given integer $k>1$ and  $m>0$, a collection  $\mathcal{T}=\{T_1,\ldots,T_k\}$ of subsets
 of a set $E$, such that each member of $\mathcal{T}$ has at least $m$
 elements and each $m$-element subset of $E$ is contained in a unique
 member of  $\mathcal{T}$, is called an $m$-partition of $E$. The
 following proposition is also from~\cite{oxley:book}.

 \begin{proposition}
   If $\mathcal{T}$ is an $m$-partition of $E$, then $\mathcal{T}$ is
   the set of hyperplanes of a paving matroid of rank $m+1$ on
   $E$. Moreover, for $r\geq 2$, the set of hyperplanes of every
   rank-$r$ paving matroid on $E$ is an $(r-1)$-partition of $E$.
 \end{proposition}

 Thus the collection of hyperplanes of a sparse paving matroid $M$ of rank $r\geq 2$ are the circuits of size $r$ together with
 the independent sets of size $r-1$ not contained in any circuit of
 size $r$. Also, because the hyperplanes of $M$ form an
 $(r-1)$-partition, any  subset $A$ of size $r-1$, that is not a
 hyperplane, (so $A$ is an independent set contained in some circuit of size $r$) is  contained in a  unique circuit of size $r$.

% When there are no hyperplanes of size $r-1$ in $M$, all closed sets of rank
% $r-1$, that is all hyperplanes, have size $r$. Because $M$ is paving,
% for $i<r-1$, all closed sets of rank $i$ have size $i$. Thus, in this case
% $M$ is a perfect matroid design. 
Any Steiner system $S(r-1,r,n)$ corresponds to a sparse paving matroid 
by taking the bases to be all sets of size $r$ not appearing as blocks of the Steiner system.
As the number of blocks in
 a $S(r-1,r,n)$ is $\frac{1}{r}\binom n {r-1}$ we see that the number 
of bases of the corresponding sparse paving matroid is  
$\binom n r-\frac{1}{r}\binom n
{r-1}$ = $\frac{n-r}{r}\binom n{r-1}$. The next result shows that this is a lower bound for the number of bases of a sparse paving matroid.  Because the Steiner systems $S(2,3,6p+1)$,  $S(2,3,6p+3)$
  (see~\cite{kirkman}) and  $S(3,4,6p+2)$,  $S(3,4,6p+4)$
  (see~\cite{hanani}) exist
  for all $p$,
  there is an infinite number of matroids that achieve our bound.

  \begin{theorem}
  Let $M$ be a rank-$r$ matroid with $n$ elements and $r\geq 1$.
  If $M$ is a sparse paving matroid then it has at least
  $\frac{n-r}{r}\binom n{r-1}$ bases. 
  \end{theorem}
   \begin{proof}
    If $r=1$, $M$ is isomorphic to either $U_{1,n}$, $U_{1,n-1}\oplus
  U_{0,1}$ by Lemma~\ref{sparse_rank_1}. Both of these matroids have at least $n-1$ bases.

  Let us suppose that $r\geq 2$. Because $M$ is paving, every subset
  of size $r-1$ is independent and, because it is sparse, the remarks preceding the theorem imply that any
  set of size $r-1$ is in at most one circuit of size $r$. Now, form
  the bipartite graph of bases and
  independent sets of size $r-1$. That is, the vertices are the
  independent sets of sizes $r$ or $r-1$ and there is an edge $(B,I)$
  if and only if the base $B$ contains the independent set  $I$. The
  degree of any independent set $I$ of size $r-1$ is at least $n-r$.
  %%(This is because any  element  e not in A, A+e is either a basis or a
  %%circuit, but of  the n-r+1 choices, at most one is a circuit) .
  So the number of edges in the bipartite graph is at least
  $(n-r) \binom n{r-1}$. As the degree of any basis $B$ in
  this graph
  is $r$, the result follows.
   \end{proof}

 Many invariants that are usually difficult to compute for a general
 matroid are easy for sparse paving matroids. For example, observe
 that if $M$ is sparse paving, all subsets of size $k<r$ are independent,
 and all subsets of size $k>r$ are spanning. On the other hand the
 subsets of size $r$ are either bases or circuit--hyperplanes. Thus,
 the Tutte polynomial of a rank-$r$ sparse matroid $M$ with $n$
 elements and $\lambda$ hyperplanes is
 \[
   T_{M} (x,y)= \sum_{i=0}^{r-1}\binom ni(x-1)^{r-i}+
                           \binom nr + \lambda(xy-x-y)+
                           \sum_{i=r+1}^{n}\binom ni(y-1)^{i-r}.
  \]

\section{Bounds for number of bases of paving matroids and sizes of multicomplexes}
\label{sec:bounds}

 In the previous section we gave a tight lower bound for the
 number of bases of a sparse paving matroid. Such a lower bound is a
 more difficult to obtain in the case of paving matroids. In
 order to find a bound for the number of bases we investigate further
  the functions $f(p,d)$ and $g(p,n)$ defined in Section~\ref{sec:Stanleyconj}.

 %%%%%%%%%%%%%%%%%%%% f(p,d)%%%%%%%%%%%%%%%%%%%%%%%%%%%%%%%%%
 \subsection{The function f(r,d)}

 We define two families of graphs. First, we define the graph
 $G_{r,d}$ to have one vertex corresponding to each monomial of degree $r$ over $d$ variables and have an edge
 $\{m,m'\}$ if and only if there exist distinct variables $x$ and $y$
 such that $m'=\frac{m}{x}y$. The second family is similar. We
 define $TG_{r,d}$ to have the same vertex set as $G_{r,d}$ and have an edge
 $\{m,m'\}$ if and only if there exist different variables $x$ and $y$
 such that $m'=\frac{m}{y}x$ or there exists a variable $y$ such that
 $m'=\frac{m}{y}$.

 Clearly, $G_{r,i}$ is an induced subgraph of
 $TG_{r,d}$ for all $0\leq i\leq d$. 
 Recall that a set $U$ of vertices dominates a set $U'$ of vertices in a graph 
 if every vertex in $U'\setminus U$ is adjacent to a vertex in $U$.
 The problem of finding $f(r,d)$
 can be translated to the problem of finding the vertex subset of
 $G_{r,d}$ of minimum size that dominates the
 vertex set of
 $G_{r-1,d}$. 
 
 For this purpose we define the {\it standard colouring} $C_{d}$ of
 $G_{r,d}$.  Let us suppose the $d$  variables are $\{x_0,\ldots,
 x_{d-1}\}$. To each variable $x_i$ we associate the colour
 $\varrho_d(X_i)=i \bmod d$ and then  we extend this colouring linearly to all
 monomials, that is, for
 a monomial $m=x_{0}^{t_{0}}\cdots x_{d-1}^{t_{d-1}}$, the value of
 $\varrho_d(m)$ is $0\,t_{0}+\ldots+(d-1)t_{d-1} \bmod d$.

 \begin{lemma}
     The standard colouring $\varrho_d$ is a proper colouring and
     $\chi(G_{r,d})\leq d$.
 \end{lemma}
 \begin{proof}
     If $(m,m')$ is an edge of
        $G_{r,d}$, then there exist $i\neq j$ such that
        $m'=\frac{m}{x_{i}}x_{j}$. So, $\varrho_d(m)-\varrho_d(m')=i-j\not\equiv 0 \pmod
        d$. Thus, $m$ and $m'$ receive different colours and $\varrho_d$ is a
        proper $d$-colouring of $G_{r,d}$.
\end{proof}

  \begin{proposition}
     The chromatic number $\chi(G_{r,d})$ equals the clique
     number  $\omega(G_{r,d})$ and both equal $d$.
 \end{proposition}
 \begin{proof}
         From the previous lemma we know that  $\chi(G_{r,d})\leq d$.
        Clearly the vertices in $\{x_{0}^{r},
        x_{0}^{r-1}x_{1},\ldots, x_{0}^{r-1}x_{d-1}\}$ form a
        clique. Thus $\omega(G_{r,d})\geq d$. But for any graph $G$ we
        have that $\omega(G)\leq \chi(G)$ and the result follows.
  \end{proof}

 Observe that in the previous proof we show that $x_{0}^{r}$ is in an
 $d$-clique. Actually, any monomial in   $G_{r,d}$ is in as many  cliques of
 size $d$ as the number of different variables in the monomial. That is,
 if $x_i\,|\,m$, then the vertices $\{\frac{m}{x_i} x_0,
        \ldots, \frac{m}{x_i} x_{d-1}\}$ form a clique. Thus, any
 colour class of a $d$-colouring of $G_{r,d}$ dominates $V(G_{r,d})$. So any colour class of a $d$-colouring of $G_{r,d}$ is a dominating independent set and
 thus, it is a maximal independent set and a minimal dominating set.

 Another important observation is that a colour class of
 a $d$-colouring of $G_{r,d}$ dominates
 the vertex subset $V(G_{r-1,d})$ in $TG_{r,d}$. This is because the neighbours of a
 monomial $m$ of rank $r-1$ in $V(G_{p,d})$ are $\{m\,x_0, \ldots, m\,x_{d-1}\}$ and 
 form a $d$-clique. So they must intersect each colour class of a $d$-colouring of $G_{r,d}$.

 We now define the function $\overline{f}(r,d)$ to be the minimum size of a
 chromatic class in the standard coloring $\varrho_d$ of $G_{r,d}$. The previous
 paragraph proves the following.

 \begin{proposition}\label{ff_bounds_f}
     For all $r\leq d$, we have $f(r,d)\leq \overline{f}(r,d)$.
 \end{proposition}

  Now, it is easy to give an upper bound for $f(r,d)$.

\begin{proposition}
   For all $r\leq d$, we have that $f(r,d)\leq \binom{r+d-1}{d-1}/d$.
 \end{proposition}
\begin{proof}
  $\binom{p+d-1}{d-1}/d$ is the average size of a colour class in a
  $d$ colouring of $G_{r,d}$.
\end{proof}

 While trying to find a formula for $\overline{f}(r,d)$, our
 computations appeared to point to the number of aperiodic  necklaces with $r$
 black beads and $d$ white
 beads, also known as the number of binary Lyndon words of length $r+d$
 and density $r$.  \emph{Binary necklaces} or  \emph{\it necklaces
 of beads with colours black and  white} are circular sequences of 0's
 and 1's, where two sequences obtained by a rotation are considered
 the same. That is, the necklaces of length $n$ are the orbits of the
 action of the  cyclic group $C_n$ on circular
 sequences of 0's and 1's of length $n$. A necklace of length $n$ is
 called {\it aperiodic} if the orbit has size $n$.

 The number of aperiodic necklaces with $n$ beads, $r$
 black and $d$   white,  is
 \[
     L_2(r,d)=\frac{1}{r+d}\sum_{k|(r+d,r)} \mu(k)\binom{(r+d)/k}{r/k},
 \]
 where $(a,b)$ denotes the greatest common divisor of  the integers
 $a$ and $b$ and $\mu$ is the
 classical M\"obius function. This formula is well known and is a typical
 example of the M\"obius inversion formula, see~\cite{van-lint+wilson:book}. In particular,
 note that when $d$ and $r$ are coprimes, the formula simplifies to $ \binom np/n$= $\binom{n-1}{p-1}/p$=$\binom{n-1}{d-1}/d$.

 \begin{theorem}
      If $n$ and $r$ are coprime, then $\overline{f}(r,n-r)$ equals
      the number of aperiodic necklaces of $n$ beads, $r$ black and
      $d=n-r$ white.
 \end{theorem}
 \begin{proof}
   Consider $\varphi$ the action of the cyclic group $C_d$ over
   $G_{r,d}$ given by 
   \[\varphi(x_{0}^{t_{0}}\cdots x_{d-1}^{t_{d-1}})
   = x_{1}^{t_{0}}\cdots x_{d-1}^{t_{d-2}}x_{0}^{t_{d-1}}.\] The
   orbits of this action correspond to necklaces with $r$ black beads
   and $d$ white beads. Variables correspond to white beads and to the
   right of the black bead corresponding to  $x_i$ we place as many
   black beads as the exponent of $x_i$, for $0\leq i\leq d-1$. The
   orbits of size $p+r=n$ correspond to aperiodic necklaces.

   Let us see the effect of $\varphi$ on the standard colouring
   $\varrho$, that is, we want to find  $\varrho(\varphi(m))$ for a
   monomial $m$=$x_{0}^{t_{0}}\cdots x_{d-1}^{t_{d-1}}$. We have
   $\varrho(\varphi(m))-\varphi(m)\equiv r \mod d$. Thus, every orbit
   has size $d$ and all the monomials in the orbit have different
   colours. We conclude that in this case the number of aperiodic
   necklaces equals  the common size of any colour class in the
   standard colouring $\varrho$ of $G_{r,d}$.
 \end{proof}

  \begin{conjecture}\label{conjecture}
       $f(r,d)=\overline{f}(r,d)= L_2(r,d)$.
  \end{conjecture}

 Notice that if $I'$ is a set
 of monomials of size $f(r,d)$ which dominates the vertices in
 $G_{r-1,d}$, it is a dominating set in $G_{r,d}$. This is because, if
 $m\in V(G_{r,d})$, then for some $x_i$ the monomial $m'=m/x_i$ is in
 $V(G_{r-1,d})$. But the set of neigbours of $m'$ in $V(G_{p,d})$
 is $S=\{m'\,x_0, \ldots, m'\,x_{d-1}\}$ and $m\in
 S$. As an element $m''$ of $I'$ has to be in $S$ and $S$ induces a
 complete graph, we conclude that $m''$ and $m$ are adjacent.

%%%%%%%%%%%%%%% g(n,r) %%%%%%%%%%%%%%%%%%%%%%%%%%%%%%%%%%%%%
\subsection{The function g(p,d)}

 For a connected paving matroid we can use the Brown-Colbourn
 Theorem~\ref{th:Brown-Colbourn}
 mentioned earlier to bound $h_r$ for $r\geq 1$ from below by
 \[
    S(r,n)= (-1)^{r-1}\sum_{i=0}^{r-1} {(-1)^{i}\binom{n-r+i-1}i}.
 \]
 A few values of $S$ are given by the following.
 \begin{proposition}\mbox{ }
 \begin{itemize}
  \item $S(1,n)=1$ for all $n\geq 1$.
  \item $S(2,n)=n-3$ for all $n\geq 2$.
  \item $S(n,n)=(-1)^{n-1}$ for all $n\geq 1$.
  \item $S(n-1,n)= n-1 \mod 2$ for $n\geq 2$.
  \item $S(n-2,n)=\lfloor \frac{n-1}{2} \rfloor$
  \end{itemize}
  \end{proposition}
  \begin{proof}\mbox{ }
    \begin{itemize}
  \item $S(1,n)=(-1)^0 \binom{n-1}0=1$ for all $n\geq 1$.
  \item $S(2,n)$ = $(-1) ( \binom{n-3}0 - \binom{n-2}1)$ = $n-3$ for all $n\geq 2$.
  \item $S(n,n)$ = $(-1)^{n-1}\sum_{i=0}^{n-1} (-1)^i\binom{i-1}i$
  = $(-1)^{n-1}$ for all $n\geq 1$. Here we adopt the usual convention that $\binom a0=1$ for
  all $a$ and $\binom ab=0$ for all integers $a$ and $b$
  such that $b>a$ and $b> 0$.
  \item $S(n-1,n)$ = $(-1)^{n-2}\sum_{i=0}^{n-2} {(-1)^i \binom ii}$ = $(-1)^{n-2}\sum_{i=0}^{n-2}{(-1)^i}$ that is 1 if $n$ is even
  and 0 otherwise.
  \item  $S(n-2,n)$ = $(-1)^{n-3}\sum_{i=0}^{n-3} {(-1)^i \binom{i+1}i}$ = $(-1)^{n-3}\sum_{i=0}^{n-3}{(-1)^i (i+1)}$. This is
  $(-1)^{n-3}(-(n-1)/2)=(n-2)/2$ if $n-2$ is even and $(n-1)/2$ if
  $n-2$ is odd.
  \end{itemize}
  \end{proof}
  The sequence $\{S(r,n)\}$ has the same  recursion  as the binomial
  coefficients.
 \begin{theorem}
  For $n\geq r\geq 1$
   \[
      S(r+1,n+1)=S(r+1,n)+S(r,n)
   \]
 \end{theorem}
 \begin{proof}
   This follows directly by the Pascal--Stifel's formula $\binom{n+1}{r+1}$= $\binom n{r+1}+\binom nr$.
  \end{proof}

This result is enough to show that the integer sequence $\{S(r,n)\}$
 is sequence A108561 
 in \cite{sloane:online}, as both satisfy the same recurrence and the
 same boundary conditions. 
 
  How does $S(r,n)$
 compare with $f(r,n-r)$? We can prove the following.
 \begin{theorem}
   For $n\geq 3$ and $r\leq n-2$ we have that $f(r,n-r)\leq S(r,n)$
  \end{theorem}
  \begin{proof}
    For $r=1$, $f(1,n-1)=1=S(1,n)$ for all $n$. For $r=n-2$, we have
    $f(n-2,2)=\lceil \frac{n-2}{2} \rceil$ = $\lfloor
    \frac{n-1}{2} \rfloor=
    S(n-2,n)$. By Lemma~\ref{lem:recursion_f(p,d)}, $f(r,n-r)\leq
    f(r,n-r-1)+f(r-1,n-r)$. Using induction this is at most
    $S(r, n-1)+S(r-1,n-1)$ which equals $S(r,n)$.
  \end{proof}

A coloopless paving matroid that is not connected must have rank one, so the previous result implies that $f(r,n-r)\leq
 g(r,n)$ for $n\geq 3$ and $2\leq r\leq n-2$.  
 Thus, we have an alternative proof of
 Corollary~\ref{Stanley_Paving} because it is easy to check the inequality for the remaining values of $r$ and $n$.

%%%%%%%%%%%%%%%%  CONCLUSIONS %%%%%%%%%%%%%%%%%%%%%%%%
  \section{Conclusion}

  We have proved Stanley's conjecture for paving matroids. This adds another
  case to the stream of result that prove the conjecture for a
  particular family of matroids~\cite{ha+stokes+zanello:O-sequences,merino:chip-firing-matroid-complexes,oh,schweig,stokes}. However,  paving
  matroids have been conjecture to be ``most'' matroids
  in~\cite{mayhew+newman+welsh+whittle} while for the other classes for which Stanley's conjecture has been established, the proportion of
  matroids with $n$ elements in each class is negligible in comparison
  with the
  total number of matroids with $n$ elements as $n$ tends to
  infinity. Also, notice that any rank three simple matroid is paving.

  The problem of giving good lower bounds on the number of bases of a
  paving matroid appears to be a challenging but interesting problem.
  The function $f(r,n-r)$ gives
  us a lower bound for the number of bases of paving matroids, but is
  not tight in most cases. When $n-r=2$ and $n$ is even,
  $f(r,n-r)$ gives the lower bound of $n(n-2)/2$ for $n$ which is
  achieve by the  dual matroid of the 2-thickening of $U_{2,m}$; when $n$ is odd it
  gives the lower bound of $(n-1)^2/2$ which is
  achieved by the  dual matroid of the free extension of the
  2-thickening of $U_{2,m}$. So, in this case $f(r,n-r)$ is a tight
  lower  bound. But for the case of $n-r=3$ the situation is quite
  different. When rank is 2, the bound gives 6 bases, which are achieve
  by the paving matroid $U_{1,3}\oplus U_{1,2}$. When the rank is 3
  the lower bound gives 13 but there are 8 coloopless paving matroids
  with 6 elements and rank 3, yet the minimum number of bases is
  15. Even if we use $S(r,n)$ for obtaining a lower bound, we get 14
  in this case.

  The function $f(p,d)$  is very intriguing. The function itself seems
  very difficult to compute from the definition. We can prove that
  $f(p,d)=\overline{f}(p,d)$ for $d=1,2,3$ and all $p\geq 1$ and also
  for $d=4$ for $1\leq p\leq 6$.  We have checked that
  $\overline{f}(p,d)= L_2(p,d)$ for many small values of $p$ and $d$
  with $(p,d)>1$  by using  Maple. 
  Conjecture~\ref{conjecture} would imply, for example that
  $f(p,d)=f(d,p)$ which
  geometrically is not so easy to see and we have been unable to prove.

\small

\setlength{\itemsep}{-.8mm}

%%% Local Variables:
%%% mode: latex
%%% TeX-master: t
%%% End:

\end{document}